\documentclass[10pt]{article}
\usepackage{amsmath,amsthm,amsfonts,latexsym,epsfig}
\newtheorem{theorem}{Theorem}[section]

\newtheorem{proposition}[theorem]{Proposition}

\newtheorem{definition}[theorem]{Definition}
\usepackage{epsfig}
\newcommand{\supp}{\mathrm{supp}}

\newcommand{\al}{\alpha}
\newcommand{\be}{\beta}
\newcommand{\ga}{\gamma}
\newcommand{\la}{\lambda}
\newcommand{\de}{\delta}
\newcommand{\De}{\Delta}
\newcommand{\ka}{\kappa}
\newcommand{\si}{\sigma}

\newcommand{\f}{\frac}
\newcommand{\tfh}{\tfrac{1}{2}}
\newcommand{\cd}{\cdots}
\newcommand{\ld}{\ldots}

\newcommand{\mcM}{\mathcal{M}}
\newcommand{\mcV}{\mathcal{V}}
\newcommand{\mcE}{\mathcal{E}}
\newcommand{\mcF}{\mathcal{F}}
\newcommand{\mcS}{\mathcal{S}}
\newcommand{\mcR}{\mathcal{R}}

\newcommand{\mcB}{\mathcal{B}}
\newcommand{\mcA}{\mathcal{A}}
\newcommand{\mcL}{\mathcal{L}}
\newcommand{\mcP}{\mathcal{P}}
\newcommand{\mcC}{\mathcal{C}}
\newcommand{\msfC}{\mathsf{C}}
\newcommand{\msfc}{\mathsf{c}}
\newcommand{\remind}[1]{{}}
\title{Annular embeddings of permutations for arbitrary genus}
\author{I.P. Goulden\footnote{Department of Combinatorics and
Optimization, University of Waterloo, email: {\small\texttt
ipgoulden@uwaterloo.ca}}~
and William Slofstra\footnote{Department of Mathematics, University of 
California, Berkeley, email: {\small\texttt slofstra@math.berkeley.edu}}}

\date{March 4, 2008}
\begin{document}
\maketitle

\begin{abstract}
In the symmetric group on a set of size $2n$, let $\mcP_{2n}$ denote
the conjugacy class of involutions with no fixed points (equivalently,
we refer to these as ``pairings'', since each disjoint cycle has
length $2$). Harer and Zagier explicitly determined the distribution
of the number of disjoint cycles in the product of a fixed cycle
of length $2n$ and the elements of $\mcP_{2n}$. Their famous result
has been reproved many times, primarily because it can
be interpreted as the genus distribution for $2$-cell embeddings
in an orientable surface, 
of a graph with a single vertex attached to $n$ loops.
In this paper we give a new formula for the cycle distribution
when a fixed permutation with two cycles (say the lengths are $p,q$,
where $p+q=2n$) is multiplied by the elements of $\mcP_{2n}$.
It can be interpreted as the genus distribution for $2$-cell embeddings
in an orientable surface,
of a graph with two vertices, of degrees $p$ and $q$. In terms
of these graphs, the formula involves a parameter that allows
us to specify, separately, the number of edges between the two vertices
and the number of loops at each of the vertices.
The proof is combinatorial, and uses a new algorithm that we
introduce to create all rooted forests containing a given rooted
forest.
\end{abstract}

\section{Introduction}\label{intro}

Let $[p]=\{1,\ldots ,p\}$, and ${\mathcal{S}}_{p}$ be the
set of permutations of $[p]$, for $p\ge 0$. When $p\ge 0$ is even,
let ${\mathcal{P}}_{p}$ be
the set of {\em pairings} on $[p]$, which are partitions of the
set $[p]$ into disjoint pairs (subsets of size $2$). We
refer to the single element of ${\mathcal{P}}_{0}$ as
the {\em empty} pairing. Where the context is appropriate, we shall also
regard ${\mathcal{P}}_{p}$ as the conjugacy class of involutions
with no fixed points in ${\mathcal{S}}_{p}$.  In this latter context, each
pair becomes a disjoint cycle consisting of that pair of elements.
Of course, the number of pairings in ${\mathcal{P}}_{p}$ is 
 $(p-1)!!=\prod_{j=1}^{\tfh p} (2j-1)$, with the empty product convention that $(-1)!!=1$.

Now, for $p>0$ and even, let $\ga_p=(1\, 2\ld p)$, in disjoint cycle notation,
and let ${\mathcal{A}}_{p} =\{ \mu\ga_p^{-1} :\mu\in{\mathcal{P}}_{p}\}$.
Let $a_{p,k}$ be the number of permutations in ${\mathcal{A}}_{p}$
with $k$ cycles in the disjoint cycle representation, for $k\ge 1$.
The generating series for these numbers are given
by $A_p(x)=\sum_{k\ge 1} a_{p,k} x^k$. Harer and Zagier~[\ref{hz}] obtained
the following result.

\begin{theorem}\label{hzthm}
{\em (Harer and Zagier~[\ref{hz}])} For a positive, even integer $p$, with $n=\tfh p$,
$$A_p(x)=\
(2n-1)!!\sum_{k\geq 1} 2^{k-1}{n\choose k-1}{x\choose k}.$$
\end{theorem}

Other proofs of Theorem~\ref{hzthm} have been given
by Itzykson and Zuber~[\ref{iz}],
Jackson~[\ref{j}], Kerov~[\ref{ke}],
Kontsevich~[\ref{k}], Lass~[\ref{l}], Penner~[\ref{p}] and Zagier~[\ref{z}] (see
also the survey by Zvonkin~[\ref{zv}], Section 3.2.7 of Lando and Zvonkin~[\ref{lz}] and the
discussion in Section 4 of the paper by Haagerup and
Thorbjornsen~[\ref{ht}]). Recently, Goulden and Nica~[\ref{gn}] gave
a direct bijective proof of Theorem~\ref{hzthm}. In the present
paper, we consider a similar bijective approach to extend this
important result of Harer and Zagier to the case
in which the permutation $\ga_p$ is replaced by a fixed
permutation with two cycles in its disjoint cycle representation.
Some additional notation is required.

Let $[q]'=\{1',\ldots ,q'\}$,
and let ${\mathcal{S}}_{p,q}$ be the
set of permutations of $[p]\cup[q]'$, for $p,q\geq 0$.
Let ${\mathcal{P}}_{p,q}$ be the set of pairings on $[p]\cup[q]'$,
for $p,q\geq 0$, where $p+q$ is even (we refer to the
single element of ${\mathcal{P}}_{0,0}$ as the {\em empty} pairing).
A pair in a pairing is called {\em mixed} if it consists of
one element from $[p]$ and one element from $[q]'$.
Where the context is appropriate, we shall also regard
 ${\mathcal{P}}_{p,q}$ as the conjugacy class of involutions with no
fixed points in ${\mathcal{S}}_{p,q}$.
For $p,q\geq 1$, we consider
the permutation $\ga_{p,q}=(1\, 2 \ldots p)(1'\, 2'\ldots q')$,
and let ${\mathcal{A}}^{(s)}_{p,q} =\{ \mu\ga_{p,q}^{-1} :\mu\in
{\mathcal{P}}_{p,q}\; \mbox{has}\; s \;\mbox{mixed pairs}\}$,
and $a^{(s)}_{p,q,k}$ be the number of permutations
in ${\mathcal{A}}^{(s)}_{p,q}$ with $k$ cycles in the disjoint
cycle representation, for $k\geq 1$.
Consider the generating series
$$A^{(s)}_{p,q}(x)=\sum_{k\geq 1} a^{(s)}_{p,q,k}x^k.$$
The main result of this paper is the following expression for $A^{(s)}_{p,q}(x)$.

\begin{theorem}\label{mainthm}
For $p,q,s\ge 1$, with $p,q,s$ of the same odd-even parity and $n=\tfh (p+q)$, we have
$$A^{(s)}_{p,q}(x) =
 p! q! \sum_{k=1}^{n+1}
\sum_{i=0}^{\lfloor\tfh p\rfloor}\sum_{j=0}^{\lfloor\tfh q\rfloor}
\frac{{x \choose k}{n-i-j \choose k-1}}{2^{i+j}i!j!(n-i-j)!}\De^{(s)}_{k,p,q},$$
where
$$\De^{(s)}_{k,p,q}=
{k-1\choose \tfh (p-s)-i}{k-1\choose \tfh (q-s)-j}-{k-1\choose \tfh (p+s)-i}{k-1\choose \tfh (q+s)-j}.$$
\end{theorem}
Note that Theorem~\ref{mainthm} gives a summation of nonnegative terms, since
for all choices of summation indices $k,i,j$ with $k-1\le n-i-j$ (so
that ${n-i-j \choose k-1}$ is nonzero), the difference $\De^{(s)}_{k,p,q}$ is nonnegative.
The proof of Theorem~\ref{mainthm} is based on a combinatorial
model that is developed in Section~\ref{sec2}. As a consequence, it is sufficient
to enumerate a particular graphical object that we call a \textit{paired array}.
We then give two combinatorial
reductions, in Sections~\ref{sec3} and~\ref{sec4}, in terms of
a simpler class of paired arrays called \textit{vertical} paired arrays.
These are explicitly enumerated in Section~\ref{sec5}, which allows
us to complete the proof of Theorem~\ref{mainthm}. One of the combinatorial
conditions on paired arrays is that two graphs associated with them must be acyclic.
Because of this, a key component of Sections~\ref{sec4} and~\ref{sec5} is
the enumeration of rooted forests which contain a given forest as a subgraph.
Thus in Section~\ref{sec35} we give a new bijection for this fundamental
combinatorial problem.
However, before we turn to our combinatorial model
and subsequent reductions, we consider some consequences of Theorem~\ref{mainthm}, and
give some comparisons to results in the existing literature.

A major reason that Harer and Zagier's result (Theorem~\ref{hzthm}) is important
(as evidenced by so many published proofs) is that it can be restated as an
equivalent geometric problem in terms of maps. A {\em map} is an embedding
of a connected graph (with loops and multiple edges allowed) in an orientable surface,
partitioning the surface into disjoint regions (called the {\em faces} of the
map) that are homeomorphic to discs (this is called a two-cell embedding). A
{\em rooted} map is a map with a distinguished edge and incident vertex (so,
the map is ``rooted'' at that end of the distinguished edge). The
well-known embedding theorem allows us to consider this as equivalent to
a pair of permutations and their product (see, e.g., Tutte~[\ref{t}], where
the terminology ``rotation system'' is used to describe this triple of permutations).
From this point of view,
the $k$th coefficient $a_{p,k}$ in
the generating series $A_p(x)$ evaluated in Theorem~\ref{hzthm} is equal to the number
of rooted maps with $1$ vertex, $n$ edges and $k$ faces (where $n=\tfh p$, as
in Theorem~\ref{hzthm}). Denoting the genus of the
surface in which such a map is embedded by $g$, then the Euler-Poincar\'e Theorem
implies that $1-n+k=2-2g$, or that $k=n-2g+1$.

Similarly, Theorem~\ref{mainthm} has a geometric interpretation.
Let $\msfC_{p,q}$ be the conjugacy class of $\mcS_{p+q}$ in which
there are two disjoint cycles, of lengths $p$ and $q$.
Then the coefficient $a^{(s)}_{p,q,k}$ in the generating series $A^{(s)}_{p,q}(x)$ is
equal to $(2n-1)!/|\msfC_{p,q}|$ times the number of rooted
maps with $2$ vertices (of degrees $p$ and $q$), $n$ edges (exactly $s$ of which join the
two vertices together, plus $\tfh (p-s)$ that are loops at the vertex of degree $p$,
plus $\tfh (q-s)$ that are loops at the vertex of
degree $q$), and $k$ faces (where $n=\tfh(p+q)$, as in Theorem~\ref{mainthm}).
In this case, if we denote the genus of the surface in which such a map is embedded by $g$,
then we obtain $k=n-2g$.

Of course, since genus is a nonnegative integer, we must have $a^{(s)}_{p,q,n+1}=0$,
and indeed the coefficient of $x^{n+1}$ in the summation for $A^{(s)}_{p,q}(x)$ given
in Theorem~\ref{mainthm} is zero, since the summand corresponding
to $k=n+1, i=j=0$  (which has ${x\choose n+1}$ as
a factor) is itself
equal to zero.  For the planar case, which corresponds to $g=0$, the only
nonzero summand that contributes to the coefficient of $x^n$ in the
summation of Theorem~\ref{mainthm} corresponds to $k=n,i=j=0$, and
this gives immediately that
$$a^{(s)}_{p,q,n}=s{p\choose\tfh (p-s)}{q\choose\tfh (q-s)}.$$
This checks with the straightforward computation that one
can make to determine this value by elementary means -- there are $s$ edges
between the two vertices; between the ends of these edges at each vertex is
an even number of vertices, joined by loops without crossings (and
there is Catalan number of such arrangements for each such even interval).
\vspace{.1in}

This explains the term ``genus'' in the title;  the term ``annular'' is
adapted from its usage in Mingo and Nica~[\ref{mn}]. It refers to an
equivalent embedding for a map with two vertices, in an annulus. The
ends of the edges incident with one of the vertices (say the one of
degree $p$) are identified with $p$ points arranged around the disc
on the exterior of the annulus, and the ends incident with the other
vertex are identified with $q$ points arranged around the disc
on the interior of the annulus. The points corresponding to the
two ends of an edge are joined by an arc in the interior of the annulus.
\vspace{.1in}

We have been able to find one relevant enumerative result (Jackson~[\ref{j}]) in the literature about
such maps, in which the total number of edges is specified, but not the exact number
joining the two vertices together. To compare this result to our
main result, we must sum
over $s\ge 1$ (since the underlying graph must be connected, then $s$, the
number of edges joining the two vertices together, must be positive), and thus
define
$$A_{p,q}(x)=\sum_{s\ge 1}A^{(s)}_{p,q}(x).$$
Then Jackson~[\ref{j}] has considered the case $p=q=n$, and obtained the
following result, restated in terms of our notation (by applying the
proportionality constant $(2n-1)!/|\msfC_{n,n}|=n$).

\begin{theorem}\label{jint}
{\em (Jackson~[\ref{j}])}
For $n\ge 1$,
$$A_{n,n}(x)=
n!\sum_{j=0}^{\lfloor \tfh (n-1) \rfloor} \sum_{i=0}^{n-2j-1}
\sum_{k=0}^{\lfloor \tfh (n-2j-1) \rfloor}
4^{-k}
{2k\choose k}{n\choose 2k}
{2j\choose j}{n-2j-1\choose i}{x+j+i\choose n}.$$
\end{theorem}

By slightly modifying Jackson's~[\ref{j}] integration argument
we are able to obtain the following expression for $A_{p,q}(x)$,
with arbitrary $p,q$ of the same parity.

\begin{theorem}\label{gsint}
For $1\le p\le q$, with $p+q$ even, and $n=\tfh (p+q)$,
$$A_{p,q}(x)=
p!q!\sum_{j=0}^{\lfloor \tfh (p-1) \rfloor} \sum_{i=0}^{n-2j-1}
\sum_{k=0}^{\lfloor \tfh (p-2j-1) \rfloor}
\f{1}{2^{n-p+2k}k!(p-2k)!(n-p+k)!}
{2j\choose j}{n-2j-1\choose i}{x+j+i\choose n}.$$
\end{theorem}

We have checked computationally, with the help of Maple, that Theorems~\ref{jint}
and~\ref{gsint} agree with Theorem~\ref{mainthm}, summed over $s\ge 1$, for
a wide range of values of $p,q$. However, we have been unable to
prove this for all $p,q$, since we have not been able to show that the
sum over $s\ge 1$ of the result of Theorem~\ref{mainthm} is
equal to the result of Theorem~\ref{gsint}. Note that the summation in
Theorem~\ref{gsint} can be made symmetrical in $p,q$ (so the ordering $p\le q$ is
not required) by changing the summation variable $k$ to $m=p-2k$.

The method employed in Jackson~[\ref{j}] for Theorem~\ref{jint}, and in many of
the papers listed above that give proofs of Theorem~\ref{hzthm},
is matrix integration. However, we do not see how to adapt the matrix
integration methodology to prove our main result, Theorem~\ref{mainthm},
since it doesn't seem possible to specify that there are exactly $s$ edges
joining the two vertices together in the matrix method. The simplicity
of our result seems to suggest that an extended theory of matrix integration
to allow a specified number of edges between particular vertices might be possible,
and worth investigating. The simplicity of the result also suggests that
there should be a more direct combinatorial proof than the one  presented in
this paper.

\section{The combinatorial model}\label{sec2}

\subsection{Paired surjections}\label{sec21}
The combinatorial model for our proof of Theorem~\ref{mainthm} is based
on a {\em paired surjection}, which has the following definition.

\begin{definition}\label{shiftsydef}
For $p,q,s,k\ge 1$, with $p,q,s$ of the same odd-even
parity, let ${\mathcal{B}}^{(s)}_{p,q,k}$ be the set of
ordered pairs $(\mu ,\phi )$, where $\mu\in{\mathcal{P}}_{p,q}$ has $s$ mixed
pairs, and $\phi$ is
a surjection from $[p]\cup[q]'$ onto $[k]$,
satisfying the condition
\begin{equation}\label{muphicond}
\phi(\mu (i))=\phi(\gamma_{p,q} (i)) \;\;
\mbox{for all} \;i\in [p]\cup[q]'.
\end{equation}
Such an ordered pair $(\mu ,\phi )$ is called a {\em paired surjection}.
Let $b^{(s)}_{p,q,k}=|\mcB^{(s)}_{p,q,k}|$.
\end{definition}

In the following result, the generating series $A^{(s)}_{p,q}(x)$  evaluated
in Theorem~\ref{mainthm} is expressed in terms of the numbers $b^{(s)}_{p,q,k}$ of
paired surjections. Paired surjections are closely related to shift-symmetric
partitions, that arose in Goulden and Nica~[\ref{gn}]. Indeed, the proof of
the following result is identical to the proof of Proposition~1.3 in
Goulden and Nica~[\ref{gn}], and is hence omitted.

\begin{proposition}\label{fallfac}
For $p,q,s\geq 1$, with $p,q,s$ of the same odd-even parity, we have
$$A^{(s)}_{p,q}(x) = \sum_{k\geq 1} b^{(s)}_{p,q,k} {x\choose k}.$$
\end{proposition}

We consider $(\mu,\phi)\in{\mathcal{B}}^{(s)}_{p,q,k}$, and construct
various objects associated with $(\mu,\phi)$. First let
 $C_{i}=\phi^{-1}(i)\cap [p]$ and $C'_{i}=\phi^{-1}(i)\cap [q]'$,
for $i\in[k]$. Let $D=\{ i:|C_i|\ge 1\}$, and $D'=\{ i:|C'_i|\ge 1\}$, 
and let $m_i=\max C_i$, $i\in D$,  and $m'_i=\max C'_i$, $i\in D'$.
Suppose that $1$ is contained in $C_a$, and that $1'$ is contained in $C'_{b}$.
Define $\psi:D\setminus\{a\}\rightarrow D$ by $\psi(i)=j$ when $\phi(\mu(m_i))=j$,
and $\psi':D'\setminus\{b\}\rightarrow D'$ by $\psi'(i)=j$ when $\phi(\mu(m'_i))=j$.

Now, if $\psi(i)=j$, then (interpreting $1$ as $p+1$) condition~(\ref{muphicond}) means
that $m_i+1\in C_j$, so we have $m_i<m_j$. This implies that the functional digraph
of $\psi$ (the directed graph on vertex-set $D$ with an arc directed
from $i$ to $\psi(i)$  for each $i\in D\setminus\{a\}\}$) is
actually a tree, in which all arcs are directed towards vertex $a$ (which we consider
as the root of this tree). We denote this rooted tree by $T$.
Similarly, the functional digraph of $\psi'$, on vertex-set $D'$, is
also a tree, with all arcs directed towards vertex $b$ (which we consider
as the root of this tree). We denote this rooted tree by $T'$.

One condition that the paired surjection $(\mu,\phi)$ satisfies is that the number of
mixed pairs containing an element of $C_i$ is equal to the number
of mixed pairs containing an element of $C'_i$ for all $i\in [k]$.
(For the reason that this necessary condition
arises, see the discussion of ``unique label recovery'' in the next section.) We
call this the {\em balance}
condition for $(\mu,\phi)$. 
The fact that $\phi$ is a surjection is equivalent to $|C_i|+|C'_i|\ge 1$, for
$i\in [k]$, and we call this the {\em nonempty} condition for $(\mu,\phi)$. 
The fact, established above, that
the graphs of $\psi$ and $\psi'$ are trees is called the {\em tree} condition
for $(\mu,\phi)$.

\subsection{A graphical model}

Now we consider a graphical representation for the
paired surjection $(\mu,\phi)$, called
its {\em labelled paired array}. This is an array of cells, arranged
in $k$ columns,
indexed $1,\ldots ,k$ from left to right, and two rows. In column $i$ of
row $1$, place an ordered list of $|C_i|$  vertices, labelled by the
elements of $C_i$ from left to right; in column $i$ of
row $2$, place an ordered list of $|C'_i|$  vertices, labelled by the
elements of $C'_i$ from left to right. For each pair of $\mu$ draw
an edge between the vertices whose labels are given by the pair.

\begin{figure}[ht]
\begin{center}
\scalebox{.6}{\includegraphics{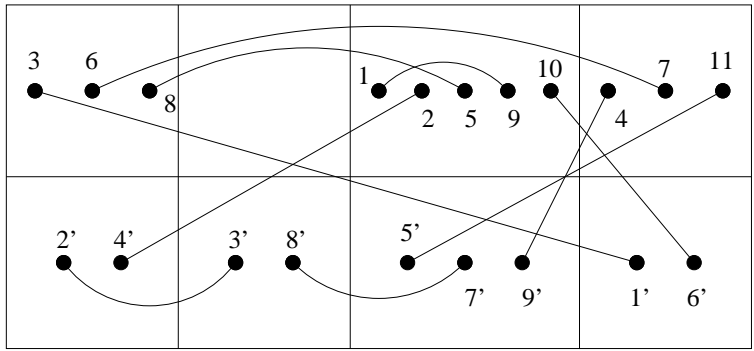}}
\end{center}
\caption{A labelled paired array.}\label{labarr}
\end{figure}

\begin{figure}[ht]
\begin{center}
\scalebox{.6}{\includegraphics{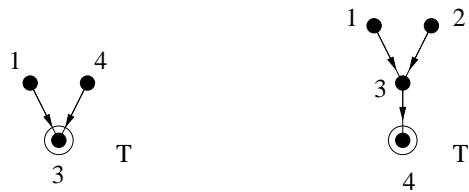}}
\end{center}
\caption{Two rooted trees.}\label{trees}
\end{figure}

\begin{figure}[!b]
\begin{center}
\scalebox{.6}{\includegraphics{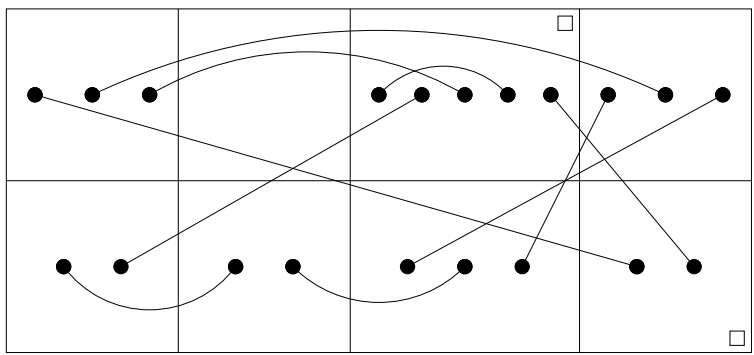}}
\end{center}
\caption{A paired array.}\label{arr}
\end{figure}

For example, when $p=11,q=9,s=5,k=4$, consider $(\mu,\phi)\in {\mathcal{B}}^{(s)}_{p,q,k}$,
given by $\mu=\{\{1,9\},\{5,8\},\{6,7\},$
$\{2',3'\},\{7',8'\},$
$\{2,4'\},\{3,1'\},\{4,9'\},\{10,6'\},\{11,5'\}\}$,
and $\phi^{-1}(1)=\{ 3,6,8,2',4'\}$, $\phi^{-1}(2)=\{ 3',8'\}$,
$\phi^{-1}(3)$ $=\{ 1,2,5,9,10,5',7',9'\}$, $\phi^{-1}(4)=\{ 4,7,11,1',6'\}$.
The corresponding labelled paired array is given
in Figure~\ref{labarr}, and the trees $T$ and $T'$ are given in Figure~\ref{trees}.

Now suppose that we mark the cells in column $a$ of row $1$ and
in column $b$ of row $2$ (by placing
a small box in the top righthand corner of the marked cell in row 1, and in
the bottom righthand corner of the marked cell in row 2), and 
remove the labels from all vertices -- call the resulting object
the {\em paired array} of $(\mu,\phi)$. The ordered list of vertices
in each cell is now to be interpreted as a generic totally ordered
set, with the given left to right order, and the pairing $\mu$ now
acts on these ordered sets in the obvious way. For example, the
paired array determined from the labelled paired array displayed
in Figure~\ref{labarr} is given in Figure~\ref{arr}.

What information have we lost when the labels are removed? The
answer, perhaps surprisingly, is that no information is lost, since
we have \emph{unique label recovery} by applying condition~(\ref{muphicond})
iteratively, as follows: for the first row, place label $1$ on the
leftmost vertex in the marked cell of row $1$; for each $i$ from $2$ to $p$,
place label $i$ on the leftmost unlabelled vertex in column $\phi(\mu(i-1))$ of
row $1$. The same process applied to the second row will place labels $1'$ to $q'$ on
the vertices in row $2$. (The reader can apply this to the paired array
in Figure~\ref{arr}, to check that indeed the labelled paired array
in Figure~\ref{labarr} is recovered in this way.)
The proof that this process always works for a paired array satisfying the
balance, nonempty and tree conditions (and the proof that these conditions are
necessary for this process to work) requires only a slight modification
of the results in Section 3 of~[\ref{gn}], and is not given here
(note that neither the functions $\psi$ and $\psi'$, nor the
trees $T$ and $T'$, depend on the labels
of the vertices, and the number of mixed pairs incident
with the vertices in
%
each cell of the paired array also does not
depend on the labels, so the balance, nonempty and tree conditions can
be checked on the paired array alone).

\subsection{Paired arrays}

This motivates us to define a \emph{paired array} in the abstract (and
not as obtained by removing the labels from a labelled paired array),
and in fact to extend it to a more general class of objects, in the
following definition.

\begin{definition}\label{pairarr}
For $p,q,s,k\ge 1$, with $p,q,s$ of the same odd-even
parity, we define $\mcP\!\mcA^{(s)}_{p,q,k}$ to be the set of arrays of
cells, arranged in $k$ columns and $2$ rows, subject to the following conditions:
\begin{itemize}
\item
Each cell contains an ordered list of vertices, so that there is a total
of $p$ vertices in the first row, and $q$ vertices in the second row.  The vertices
are paired (in the language of graph theory, there is a perfect matching on the vertices),
so that $s$ pairs join a vertex in the first row to a vertex in the second row (these
are the \emph{mixed pairs}). The number of mixed pairs containing a vertex in
column $i$ of row $1$ is equal to the number of mixed pairs containing a vertex in
column $i$ of row $2$, for all $i=1,\ld ,k$ (this is called the \emph{balance} condition).
\item
There is at least one marked (with a small box) cell in row $1$, and we denote the
set of such columns by $R$. There
is at least one marked (with a small box) cell in row $2$, and we denote the
set of such columns by $R'$.
There is at least one vertex in every column that is not
contained in $R\cup R'$ (this is called the \emph{nonempty} condition).
\item
Denote the set of columns in which there is at least one vertex in
row $1$ by $D$,
and the set of columns in which there is at least one vertex in 
row $2$ by $D'$. Define the function $\psi:D\setminus R\rightarrow D$ as
follows: if the rightmost vertex in column $i$ of row $1$ is paired
with a vertex in column $j$, then $\psi(i)=j$. Similarly,
define  $\psi':D'\setminus R'\rightarrow D'$ as
follows: if the rightmost vertex in column $i$ of row $2$ is paired
with a vertex in column $j$, then $\psi'(i)=j$. The functional digraph
of $\psi$ is a forest with $|R|$ components (called the \emph{rightmost forest
for row} $1$); each component is
a tree in which all edges are directed towards an element of $R$ (and this
is called the \emph{root} of that tree).
The functional digraph
of $\psi'$ is a forest with $|R'|$ components (called the \emph{rightmost forest
for row} $2$); each component is
a tree in which all edges are directed towards an element of $R'$ (and this
is called the \emph{root} of that tree). Together, these specify 
the \emph{forest} condition. 
\end{itemize}
The elements of $\mcP\!\mcA^{(s)}_{p,q,k}$ are called \emph{paired arrays}.
A paired array is defined to be \emph{canonical} if $|R|=|R'|=1$.
Define $\mcC^{(s)}_{p,q,k}$ to be the set of canonical
paired arrays in $\mcP\!\mcA^{(s)}_{p,q,k}$, and $c^{(s)}_{p,q,k}=|\mcC^{(s)}_{p,q,k}|$.
\end{definition}

The uniqueness of label recovery described in the previous section proves that
there is a bijection (via labelled paired arrays) between the set $\mcB^{(s)}_{p,q,k}$ of
paired surjections
and the set $\mcC^{(s)}_{p,q,k}$ of
canonical paired arrays, so we have
\begin{equation}\label{equivbe}
b^{(s)}_{p,q,k}=c^{(s)}_{p,q,k}.
\end{equation}
(It is straightforward to verify that the conditions for canonical paired
arrays imply that \emph{every} column is nonempty.)
In this paper we shall determine $b^{(s)}_{p,q,k}$, and
hence the generating series~$A^{(s)}_{p,q}(x)$ via Proposition~\ref{fallfac}, by giving
a combinatorial reduction for canonical paired arrays, thus directly
determining $c^{(s)}_{p,q,k}$.

\section{Removing redundant pairs and minimal paired arrays}\label{sec3}

A {\em redundant} pair in a paired array is a vertex pair
that is \emph{not} mixed, and does \emph{not} contain
a vertex that is rightmost in an unmarked cell.
A {\em minimal} paired array is a paired array without redundant pairs.
We define $\mcM^{(s)}_{p,q,k}$ to be the set of minimal, canonical paired arrays
in $\mcP\!\mcA^{(s)}_{p,q,k}$, and~$m^{(s)}_{p,q,k}=|\mcM^{(s)}_{p,q,k}|$.
In our next result, we remove redundant pairs from a canonical paired array,
and thus show that
the enumeration of canonical paired arrays can be reduced to the
enumeration of minimal, canonical paired arrays.

\begin{theorem}\label{redundmin}
For $p,q,s,k\ge 1$, with $p,q,s$ of the same odd-even parity, we have
$$c^{(s)}_{p,q,k}=\sum_{i,j\ge 0}{p\choose 2i}(2i-1)!!{q\choose 2j}(2j-1)!!\,
m^{(s)}_{p-2i,q-2j,k}.$$
\end{theorem}

\begin{proof}
For the proof, it is convenient to introduce some notation.
A {\em partial} pairing on $[p]$ is a
pairing on a set $\alpha\subseteq\, [p]$ of
even cardinality. If $\vert\alpha\vert=2i$, then we
also call it an $i$-partial pairing.
For each of these partial pairings $\mu$, we call $\alpha$ the {\em support},
and denote this by $\supp (\mu ) = \al$.
Let ${\mathcal{R}}_{p,i}$ be
the set of $i$-partial pairings on $[p]$.
Similarly, let ${\mathcal{R}}'_{q,i}$ be the set of $i$-partial pairings on $[q]'$.

Consider an arbitrary $\al\in\mcC^{(s)}_{p,q,k}$. 
We now describe a construction for three objects, $\mu_1$, $\mu_2$ and $\be$, obtained
from $\al$. We begin by attaching
the numbers $1,\ldots,p+1$ to the vertices and the small box in row $1$ of
$\al$, from left to right (under the interpretation that all vertices in column $i$
are to the left of all vertices in column $j$ for $i<j$, and that the small box representing
a marking is rightmost in its cell). Let $\mu_1$ be
the partial pairing consisting of pairs of numbers
attached to the redundant pairs in row  $1$ of $\al$. We follow the analogous procedure
for row $2$: we attach primed numbers $1',\ldots,(q+1)'$ to the vertices and 
small box in row $2$, and let $\mu_2$ be the partial pairing consisting of the pairs of
(primed) numbers attached to the redundant pairs in row $2$ of $\al$. 
Third, we remove all redundant pairs (both vertices and edges) from $\al$,
to get the paired array $\be$, with the same marked cells as $\al$. 
The vertices in each cell of $\be$ have the same relative order as they did
in $\al$.
For example, if $\al$ is the paired
array in Figure~\ref{arr}, then we have $\mu_1=\{\{2,11\},\{ 4,7\}\}$,
$\mu_2=\{\{1',3'\}\}$, and $\be$ is given in Figure~\ref{newminarr}.

\begin{figure}[ht]
\begin{center}
\scalebox{.6}{\includegraphics{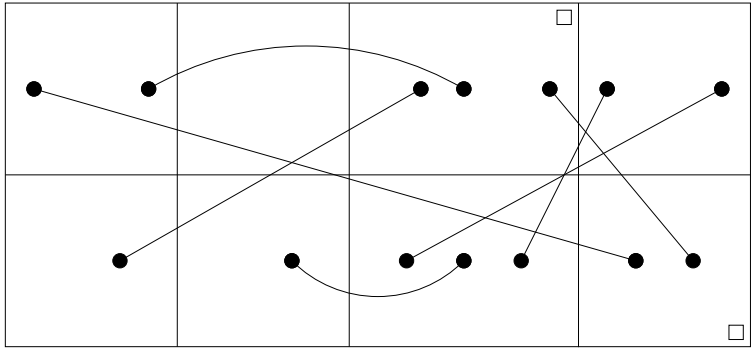}}
\end{center}
\caption{A minimal paired array.}\label{newminarr}
\end{figure}

Now, the only vertex that can be numbered $p+1$ in row $1$ is the
rightmost vertex of the rightmost nonempty cell in row $1$ (if this
cell is not marked), but this vertex cannot appear in a redundant
pair since it is rightmost in an unmarked cell. This implies that the numbers
on redundant pairs in row $1$ all fall in the range $1,\ldots ,p$,
and so $\mu_1$ is a partial pairing on $[p]$. Similarly, $\mu_2$ is
a partial pairing on $[q]'$. 
Also, since the redundant pairs that were removed in the construction do not involve the
rightmost vertex in any nonempty cell, $\be$ has the same
rightmost functions $\psi$ and $\psi'$ as $\al$, and the same mixed pairs
as $\al$, so it must satisfy the balance, nonempty and forest
conditions, which implies that $\be$ is a minimal paired array.
Thus we have a mapping
\begin{equation*}
\xi:\mcE^{(s)}_{p,q,k}\rightarrow\cup_{i,j\ge 0}
\mcR_{p,i}\times\mcR'_{q,j}\times\mcM^{(s)}_{p-2i,q-2j,k}:
\al\mapsto (\mu_1,\mu_2,\be).
\end{equation*}

We now prove that $\xi$ is a bijection.
It is sufficient to describe the inverse mapping, so that we can uniquely
recover $\al$ from an arbitrary triple $(\mu_1,\mu_2,\be)$ $\in\cup_{i,j\ge 0}
\mcR_{p,i}\times\mcR'_{q,j}\times\mcM^{(s)}_{p-2i,q-2j,k}$.
Given $(\mu_1,\mu_2,\be)$, let $\si_i=\supp (\mu_i)$, $i=1,2$,
and $\rho_1=[p+1]\setminus\si_1$, $\rho_2=[q+1]'\setminus\si_2$.
Number the vertices and small box in row $1$ of $\be$ with the elements of
$\rho_1$. Then insert vertices numbered with the elements of $\si_1$,
so that the numbers on all vertices and the small box in row $1$ increase
from left to right, and so that the vertex numbered with $l \in \si_1$
is placed in the same cell as either the vertex or small box numbered
with $\min \{ i \in \rho_1 : i > l \}$. Inserting vertices 
numbered from $\si_2$ in row $2$ using an analogous process,
pairing the inserted vertices with $\mu_1$ and $\mu_2$, and
removing the numbers, we arrive
at a paired array $\al$. It is straightforward to check that
$\al$ satisfies the balance, nonempty, and tree conditions, and
that the process described above reverses the numbering 
scheme used in the mapping $\xi$.  
Thus, we have described the mapping $\xi^{-1}$, and our proof that $\xi$ is
a bijection is complete.
The result follows from the easily established
facts that~$|\mcR_{p,i}|={p\choose 2i}(2i-1)!!$,~$|\mcR'_{q,j}|={q\choose 2j}(2j-1)!!$.
\end{proof}

\section{A bijection for rooted forests}\label{sec35}

In this section, we detour to consider the basic combinatorial question of
how many rooted forests with a given set of root vertices contain a
given rooted forest. We give a bijection for this that differs from
the standard ones in the literature, like the Cycle Lemma (see, \textit{e.g.},
[\ref{stan}], p. 67) or the Pr\"ufer Code (see, \textit{e.g.},
[\ref{stan}], p. 25), because it is more convenient for our constructions
involving paired arrays. 

Suppose we have a rooted forest $F$ (all edges directed towards a root vertex in each
component) on vertex-set $[k]$, whose components are the rooted
trees $T_1,\ld ,T_{m+n}$, $m,n\ge 1$. Suppose the root vertex
of $T_{j}$ is  $r_j$, $j=1,\ld ,m$, and the root vertex of $T_{m+j}$ is $s_j$, $j=1,\ld ,n$.
For convenience, we order the trees so that $r_1<\cdots <r_m$, $s_1<\cdots <s_n$, and
we let $S$ denote the union of the sets of vertices in the trees $T_{m+1},\ldots ,T_{m+n}$.

\begin{theorem}[Forest Completion Theorem]\label{FCT}
There is a bijection between $[k]^{m-1}\times S$ and the set of rooted forests on
vertex-set $[k]$ with root vertices $s_1,\ldots ,s_n$ that contain $F$ as a subforest.
\end{theorem}

\begin{proof}
We describe such a mapping, which we call the ``Forest Completion Algorithm'' (FCA).
Consider the $m$-tuple $a=(a_1,\ld,a_{m})\in [k]^{m-1}\times S$. We construct the forest
corresponding to $a$ iteratively, in $m+1$ stages $0,1,\ld ,m$.
At every stage, we have a forest $G$ containing $F$ as a subforest,
a permutation $\pi$ of $[m]$, and a sequence $b=(b_1,\ld ,b_m)$ in $[k]^{m}$.
Initially, at stage $0$, we have $G=F$, $\pi$ is the identity permutation
and $b=a$.  Then, for $i=1,\ld ,m$:
\begin{itemize}
\item
if $b_i$ is in
a different component of $G$ from $r_i$,
then add an arc directed from $r_i$ to $b_i$ in $G$, and
leave $\pi$ and $b$ unchanged;
\item 
otherwise (so $b_i$ is in the same component
of $G$ as $r_i$), add an arc directed from $r_i$ to $b_m$ in $G$ (to obtain the new $G$),
switch $\pi(i)$ and $\pi(m)$ in $\pi$, and switch $b_i$ and  $b_m$ in $b$. 
\end{itemize}
The forest corresponding to the $m$-tuple $a$ is the terminating forest $G$. We call the terminating
permutation $\pi$ the ``Forest Completion Permutation'' (FCP).
The significance of the FCP is that it identifies precisely the arcs
that are added to $F$ -- they are $(r_i,a_{\pi(i)})$, $i=1,\ld ,m$. In our
examples throughout the paper, we shall specify the second line in the two
line representation of $\pi$ -- the list of images $(\pi(1),\ld ,\pi(m))$.

In Figure~\ref{foralgo} we give an example of the FCA
with $k=9$, $m=3$, $n=2$. The trees $T_1$, $T_2$, $T_3$, with $r_1=2$, $r_2=4$, $r_3=7$, are
given in the box at the top left; the trees $T_4$, $T_5$, with $s_1=6$, $s_2=8$, are
given in the box at the top right. Then, corresponding to the triple $a=(9,2,3)$, we
construct the forest at the bottom
of Figure~\ref{foralgo}.
The corresponding FCP is $(3,2,1)$.
\begin{figure}[ht]
\begin{center}
\scalebox{.6}{\includegraphics{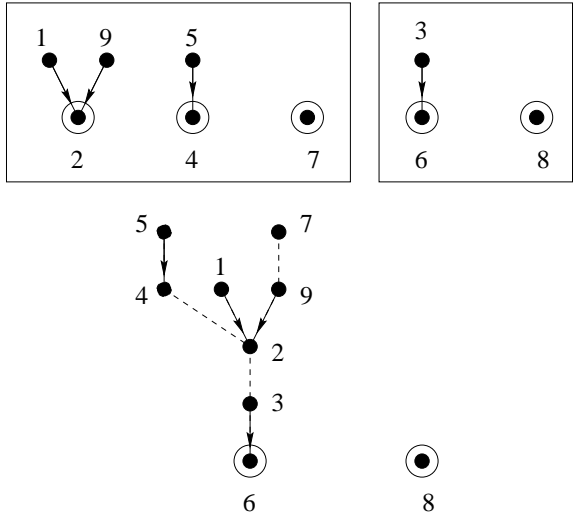}}
\end{center}
\caption{A rooted forest and subforest.}\label{foralgo}
\end{figure}

In analyzing this mapping, it is convenient to use the term ``safe'' to describe a vertex
in a component of a forest rooted at one of the vertices $s_1,\ld ,s_n$. Thus, initially,
$b_m$ is safe. It is trivial to prove by induction that, after
stage $i$, for $i=1,\ld ,m-1$, $G$ is a forest with root vertices $r_{i+1},\ld ,r_m,s_1,\ld ,s_n$,
and that $b_m$ is safe for $G$ (which implies that $r_{i+1},\ld ,r_m$ are in different components
of $G$ from $b_m$). Thus, at stage $m$, $b_m$ is indeed in a different component 
of $G$ from $r_m$, so we successfully add the final arc from $r_m$ to $b_m$,
to obtain a terminating forest $G$ rooted at $s_1,\ld ,s_n$ (this explains our use of ``safe'' -- for
extending our forest, it is always safe to add the new arc directed to $b_m$).
This proves that the FCA does indeed produce a rooted forest on vertex-set $[k]$ with
root vertices $s_1,\ld ,s_n$ that contains $F$ as a subforest.

To prove that the FCA is a bijection, we describe its inverse. Suppose that we are
given a rooted forest $H$ on vertex-set $[k]$, with root vertices $s_1,\ld ,s_n$, and that
we wish to remove the arcs directed from  $r_i$ to $c_i$, $i=1,\ld ,m$, where $r_1,\ld ,r_m$ are
distinct non-root vertices, with the convention that $r_1<\cdots <r_m$.
We proceed iteratively, through stages $m,\ld,1,0$.
At every stage we have a subforest $G$ of $H$, a permutation $\si$ of $[m]$, and a
sequence $b=(b_1,\ld ,b_m)$ in $[k]^m$. Initially, at stage $m$, we
have $G=H$, $\si$ is the identity permutation,
and $b=c$. Then, for $i=m-1,\ld ,0$, remove the arc
from $r_{i+1}$ to $b_{i+1}$ in $G$ (to get the new $G$), and:
\begin{itemize}
\item
if $b_m$ is safe for (the new) $G$, leave $\si$ and $b$ unchanged;
\item
otherwise (so $b_m$ is not safe for $G$), switch $\si_{i+1}$ and $\si_m$ in $\si$,
and switch $b_{i+1}$ and $b_m$ in $b$.
\end{itemize}
We claim that the $m$-tuple corresponding to the forest $H$ is the terminating $m$-tuple $b$,
so that this mapping uniquely reverses the FCA. In fact, it is easy to establish
that this mapping uniquely reverses the FCA stage by stage,
since it is trivial to prove by induction that the values
of $b$ and $G$ after stage $i$ of the above mapping are exactly the same as $b$ and $G$ after
stage $i$ of the FCA. It is also easy to prove that the FCP is given by $\si^{-1}$ for the
terminating $\si$.
 
The result follows, since the FCA is a bijection between the required sets.
\end{proof}

Of course, it is an immediate enumerative consequence of Theorem~\ref{FCT} that there
are $k^{m-1}|S|$ rooted forests on
vertex-set $[k]$ with root vertices $s_1,\ldots ,s_n$ that contain $F$ as a subforest.
In the special case $m+n=k$ (so that $F$ has no edges) this gives
the classical result that there are $k^{k-n-1}n$ rooted forests on vertex-set $[k]$,
with a prescribed set of $n$ root vertices (see, \textit{e.g.},~[\ref{stan}], p. 25).

\section{Removing non-mixed pairs and vertical paired arrays}\label{sec4}

A {\em vertical} paired array is a paired array in which all
pairs are mixed.
We define $\mcV^{(s)}_{k,i,j}$ to be the set of vertical paired arrays
in $\mcP\!\mcA^{(s)}_{s,s,k}$, in which there are $i+1$ marked cells
in row $1$ and $j+1$ marked cells in row $2$, for $i,j\ge 0$,
and let~$v^{(s)}_{k,i,j}=|\mcV^{(s)}_{k,i,j}|$.

In our next result,
we remove non-mixed pairs from a minimal, canonical paired array,
and thus show that
the enumeration of minimal, canonical paired arrays can be reduced to the
enumeration of vertical paired arrays.
We use the following notation. For a finite set $X$, let $\mcL_{X,i}$ denote
the set of $i$-tuples consisting of $i$ distinct elements of $X$.
Thus $\vert \mcL_{X,i}\vert=(x)_i$, where $x=|X|$ and $(x)_i$ is
the \textit{falling factorial}: for positive integers $i$, $(x)_i=x(x-1)\cd (x-i+1)$;
for $i=0$, $(x)_i=1$; otherwise $(x)_i=0$.

\begin{theorem}\label{nonmixedprop}
    For $p,q,s\ge 1$ of the same odd-even parity, let $i=\tfh (p-s)$ and $j=\tfh (q-s)$. Then
    \begin{equation*}
        m_{p,q,k}^{(s)} = \left(p\right)_i
            \left(q\right)_j
            v_{k,i,j}^{(s)}.
    \end{equation*}
\end{theorem}

\begin{proof}
Note that every element of $\mcM^{(s)}_{p,q,k}$ has exactly $i$ and $j$ non-mixed
pairs in the top and bottom rows, respectively. 
Taking an arbitrary $\al\in\mcM^{(s)}_{p,q,k}$,
we now describe a mapping that is initially identical to that used in the
proof of Theorem~\ref{redundmin}. We attach the numbers $1,\ld ,p+1$ to
the vertices and the small box in row $1$, using the same left to
right convention as in the proof of Theorem~\ref{redundmin}.
Let the pairs of numbers attached to the non-mixed pairs
in row $1$ be denoted by $(u_1,v_1),\ld ,(u_i,v_i)$, where $u_1,\ld ,u_i$ are
attached to the rightmost vertices in these
pairs (with $u_1<\cd <u_i$), and $v_1,\ld ,v_i$ are
attached to the other (not rightmost in their cell) vertices in these pairs. 

Suppose that the marked cell in row $1$ is in column $m$, and 
that the rightmost tree for row $1$ of $\al$ is $T$, so $T$ is rooted at vertex $m$. 
Now run the inverse of the FCA on $T$, to remove the arcs directed
from $\msfc(u_{\ell})$ to $\msfc(v_{\ell})$, $\ell =1,\ld ,i$, (here, $\msfc(\ell )$ denotes
the column in which the number $\ell$ appears) and let $\rho$ be the corresponding FCP.
Let $\ka_1=(v_{\rho^{-1}(1)},\ld ,v_{\rho^{-1}(i)})$.

We follow the analogous procedure
for row $2$: we attach primed numbers $1',\ldots,(q+1)'$ to the vertices and
small box in row $2$. Let the pairs of numbers attached to the non-mixed pairs
in row $1$ be denoted by $(x'_1,y'_1),\ld ,(x'_j,y'_j)$, where $x'_1,\ld ,x'_j$ are
attached to the rightmost vertices in these
pairs (with $x_1<\cd <x_j$), and $y'_1,\ld ,y'_j$ are
attached to the other (not rightmost in their cell) vertices in these pairs.

Suppose that the marked cell in row $2$ is in column $n$, and
that the rightmost tree for row $2$ of $\al$ is $T'$, so $T'$ is rooted at vertex $n$.
Now run the inverse of the FCA on $T'$, to remove the arcs directed
from $\msfc(x'_{\ell})$ to $\msfc(y'_{\ell})$, $\ell =1,\ld ,i$,
and let $\tau$ be the corresponding FCP.
Let $\ka_2=(y'_{\tau^{-1}(1)},\ld ,y'_{\tau^{-1}(j)})$.

Finally, we mark the cells in row $1$ containing $v_1,\ld ,v_i$ (in addition
to the existing marked cell in column $m$), and the cells in
row $2$  containing $y'_1,\ld ,y'_j$ (in addition to the existing
marked cell in column $n$), and then 
remove all non-mixed pairs (both vertices and edges) from $\al$,
to get the vertical paired array $\be$.
The vertices in each cell of $\be$ have the same relative order as they did
in $\al$.

\begin{figure}[ht]
\begin{center}
\scalebox{.6}{\includegraphics{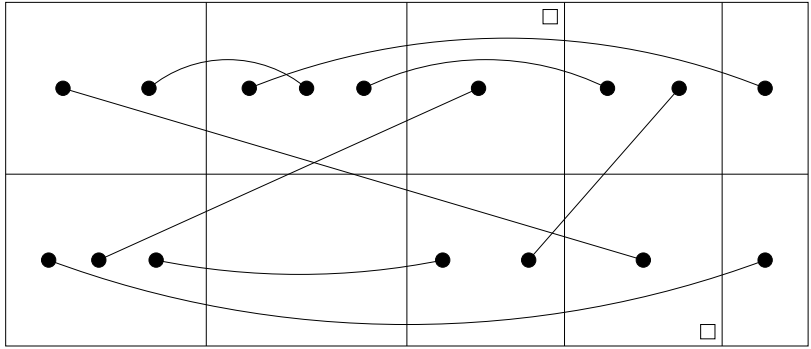}}
\end{center}
\caption{A minimal paired array.}\label{minarr}
\end{figure}

\begin{figure}[ht]
\begin{center}
\scalebox{.6}{\includegraphics{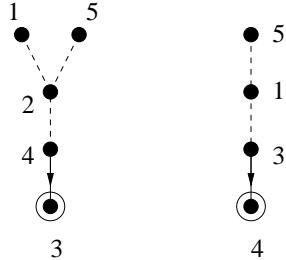}}
\end{center}
\caption{The rightmost trees for Figure~\ref{minarr}.}\label{newtrees}
\end{figure}

\begin{figure}[ht]
\begin{center}
\scalebox{.6}{\includegraphics{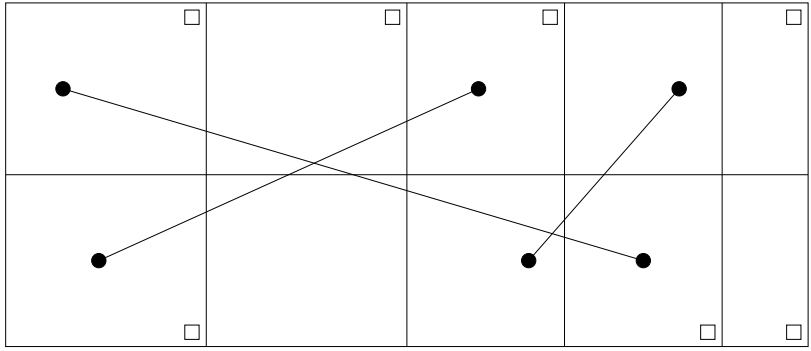}}
\end{center}
\caption{The vertical paired array for Figure~\ref{minarr}.}\label{vertarr}
\end{figure}

For example, suppose that $\al$ is the paired
array in Figure~\ref{minarr}. Then  we have $i=3$, with $u_1=2$, $u_2=5$, 
$u_3=10$, $v_1=4$, $v_2=8$, $v_3=3$, and $j=2$, with $x'_1=3'$, $x_2=8'$,
$y'_1=4'$, $y'_2=1'$. We have $m=3$, $n=4$, and the trees $T$ and $T'$ are
given in Figure~\ref{newtrees}. When we run the inverse of the FCA on $T$ to
remove the arcs $(1,2)$, $(2,4)$ and $(5,2)$ (for which we use dashed lines in
Figure~\ref{newtrees}), we obtain $(1,3,2)$ as
the FCP,  which gives $\ka_1=(4,3,8)$. When we run the inverse of
the FCA on $T'$ to remove the arcs $(1,3)$ and $(5,1)$ (for which we use dashed lines in
Figure~\ref{newtrees}),
we obtain $(2,1)$ as the FCP,  which gives $\ka_2=(1',4')$. The
vertical paired array $\be$ in this example is given in Figure~\ref{vertarr}.

Now, for the same reasons as in the proof of Theorem~\ref{redundmin}, 
we have $1\le v_{\ell}\le p$ for $\ell=1,\ld ,i$, so $\ka_1\in\mcL_{[p],i}$,
and similarly, $\ka_2\in\mcL_{[q]',j}$. Thus we have a mapping
\begin{equation*}
\zeta:\mcM^{(s)}_{p,q,k}\rightarrow
\mcL_{[p],i}\times\mcL_{[q]',j}\times\mcV^{(s)}_{k,i,j}:
\al\mapsto (\ka_1,\ka_2,\be).
\end{equation*}

In fact, $\zeta$ is a bijection. Before proving this, for $\zeta(\al)=(\ka_1,\ka_2,\be)$,
we first note the key dependencies between  $\al$ and $(\ka_1,\ka_2,\be)$ that
follow from the FCA: 
let $F$ and $F'$ denote the rightmost forests of $\be$ for rows $1$ and $2$, respectively.
Then, in $\al$ with numbers attached as in the construction above, the column containing
the vertex whose number is the last entry of $\ka_1$ is contained in the component
of $F$ rooted at $m$. Similarly, the column containing
the vertex whose number is the last entry of $\ka_2$ is contained in the component
of $F'$ rooted at $n$. For example, for $\al,\ka_1,\ka_2,\be$ given in
Figures~\ref{minarr}--\ref{vertarr}, the last entries of $\ka_1$ and $\ka_2$ are $8$ and $4'$,
respectively, corresponding to vertices of $\al$ in columns $4$ and $3$, respectively. 
Now, the forests $F$ and $F'$ in this case are obtained from the trees $T$ and $T'$, respectively,
by removing the dashed edges in Figure~\ref{newtrees}. Then, indeed, vertex $4$ is
contained in the component of $F$ rooted at $m=3$, and vertex $3$ is contained
in the component of $F'$ rooted at $n=4$.

We now prove that $\zeta$ is a bijection
by describing the inverse mapping,
so that we can uniquely recover $\al$ from an arbitrary
triple $(\ka_1,\ka_2,\be)\in\mcL_{[p],i}\times\mcL_{[q]',j}\times\mcV^{(s)}_{k,i,j}$.
Let $F$ denote the rightmost forest of $\be$ for row $1$.
Let $\{\ka_1\}$ denote the set consisting of the entries
in $\ka_1$, and let $p=s+2i$, $\de_1=[p+1]\setminus\{\ka_1\}$.
Next (as a generalization of the procedure for the inverse
of $\xi$ described in the proof of Theorem~\ref{redundmin}), we number
the vertices and small boxes in row $1$ of $\be$ with the elements of $\de_1$.
Then insert vertices numbered with the elements of $\{\ka_1\}$, so that
the numbers on all vertices and small boxes increase from left to right (with
any small box regarded as the rightmost object in its cell) and so that
the the vertex numbered with $l\in\{\ka_1\}$ is placed in the same
cell as the object numbered with $\min\{t\in\de_1: t>l\}$.  
Now suppose that $\ka_1=(w_1,\ld ,w_i)$, and use the key dependency noted above: 
let the column containing
the vertex numbered $w_i$ be contained in the component of $F$ rooted
at vertex $m$. Let $u_1<\cd<u_i$ denote the numbers attached to the
small boxes that are not in column $m$ (the columns containing $u_1,\ld,u_i$ are
the root vertices for the components of $F$ not rooted at $m$). Now, apply
the FCA on $i$-tuple $(\msfc(w_1),\ld ,\msfc(w_i))$, to give
the tree $T$ rooted at $m=\msfc(w_i)$, that contains the forest $F$ as a
subforest. Let $\rho$ be the FCP. Finally, replace the small
boxes numbered $u_1,\ld,u_i$ by rightmost vertices (in the same cells)
numbered $u_1,\ld,u_i$, pair the vertex numbered $u_{\ell}$ with the
vertex numbered $w_{\rho(\ell)}$, $\ell=1,\ld ,i$, and remove the
numbers from row $1$. Repeat the analogous process for row $2$, and
we arrive at a minimal paired array $\al$. It is straightforward to
check that the process described above reverses $\zeta$. Thus, we
have described $\zeta^{-1}$, and our proof
that $\zeta$ is a bijection is complete.
The result follows immediately.
\end{proof}

\section{Enumeration of vertical paired arrays}\label{sec5}

For every column of a vertical paired array, the
cells in rows $1$ and $2$ have the same number of vertices, because
of the balance condition. A \textit{full}, vertical paired array is
a vertical paired array with a positive number of vertices in every
column. Let $\mcF^{(s)}_{k,i,j}$ be the set of full, vertical paired arrays
in $\mcV^{(s)}_{k,i,j}$, and $f^{(s)}_{k,i,j}=\vert\mcF^{(s)}_{k,i,j}\vert$.
In Theorem~\ref{nonemvert} we shall give an explicit construction for
the elements of $\mcF^{(s)}_{k,i,j}$, and thus obtain an explicit
formula for $f^{(s)}_{k,i,j}$. 
To help in the proof of this result, we first introduce some
terminology and notation associated with an arbitrary $\al\in\mcF^{(s)}_{k,i,j}$.
A vertex is said to be \textit{dependent} if it is paired with the rightmost vertex
of an unmarked cell (in the other row).
If the rightmost vertex of an unmarked cell in row $1$, column $i$ is paired
with the rightmost vertex of an unmarked cell in row $2$, column $j$,  then we
call this a \textit{shared pair} of $\al$. In this case, the rightmost
forest $F$ for row $1$ of $\al$ contains arc $(i,j)$ and the  rightmost
forest $F'$ for row $2$ of $\al$ contains arc $(j,i)$, and we call each
of these a \textit{shared arc}. 

Now, canonically number the vertices in row $1$ of $\al$ $1,\ld ,s$, from left to right, and
number the vertices in row $2$ of $\al$ $1',\ld ,s'$, from left to right.
Let $E$ be the subforest of $F$ with only the shared arcs of $F$.
Suppose
that $F$ has $n\ge 1$ non-shared arcs, corresponding to
pairs $(x_1,y_1'),\ld,(x_n,y_n')$, where $x_1<\cd <x_n$, and $x_1,\ld ,x_n$ are
rightmost vertices in their (unmarked) cells.
Run the inverse of the FCA on the forest $F$, to
obtain the subforest $E$ by removing
the non-shared arcs directed from $\msfc(x_{\ell})$ to $\msfc(y_{\ell}')$, $\ell=1,\ld ,n$,
and let $\tau$ be the corresponding FCP. Define $a'=y'_{\tau^{-1}(n)}$.
If all arcs of $F$ are shared, then let $a'$ be the vertex in row $2$ that
is paired with the rightmost non-dependent vertex in row $1$ (we call this
the \textit{non-FCA option}). In both cases,
define $A=\msfc(a')$, and
let $\rho_0=A, \rho_1,\ld ,\rho_l$ be the vertices on the unique directed path in $E$ from vertex $A$ to
the root vertex ($\rho_l$) of the component of $E$ containing $A$. Thus $l\ge 0$, and
the cell in row $2$, column $\rho_0$ is marked, and the cell in row $2$, column $\rho_{\ell}$ is
not marked, $\ell=1,\ld ,l$. Also, the cell in row $1$, column $\rho_{\ell}$ is
not marked, $\ell=0,\ld ,l-1$.  Now define $E'$ to be the subforest of $F'$ containing
the arcs $(\rho_{\ell},\rho_{\ell+1})$, $\ell=0,\ld ,l-1$. Suppose that $F'$ has $m$ arcs
that are not in $E'$,  corresponding to
pairs $(w_1',z_1),\ld,(w_m',z_m)$, where $w_1<\cd <w_m$, and $w_1',\ld ,w_m'$ are
rightmost vertices in their (unmarked) cells. 
Run the inverse of the FCA on the forest $F'$, to
obtain the subforest $E'$ by removing
the arcs directed from $\msfc(w_{\ell}')$ to $\msfc(z_{\ell})$, $\ell=1,\ld ,m$,
and let $\ka$ be the corresponding FCP. Define $b=z_{\ka^{-1}(m)}$.
If $F'=E'$, let $b$ be the vertex in row $1$ that is paired with the rightmost
non-dependent vertex in row $2$ (again, we call this the \textit{non-FCA option}).
Let $\rho=(\rho_0,\rho_1,\ld ,\rho_l)$, which we call the \textit{tail} of $\al$.
The tail length is $l$.
We say that $b$ is \textit{in the tail} when the column containing vertex $b$ is 
one of $\rho_0,\rho_1,\ld ,\rho_l$. The \textit{type} of $\al$ is given
by $(l,\rho,a',b)$. If the cells in rows $1$ and $2$ of column $\ell$ in $\al$ have
 $\la_{\ell}$ vertices, $\ell=1,\ld ,k$, then we say that $\al$ has
\textit{shape} $\la=(\la_1,\ld ,\la_k)$. Note that $\la_1+\cd +\la_k=s$, and
that $\la_{\ell}$ is positive for
all $\ell=1,\ld ,k$, so $\la$ is a \textit{composition} of $s$ with $k$ parts.

\begin{figure}[ht]
\begin{center}
\scalebox{.6}{\includegraphics{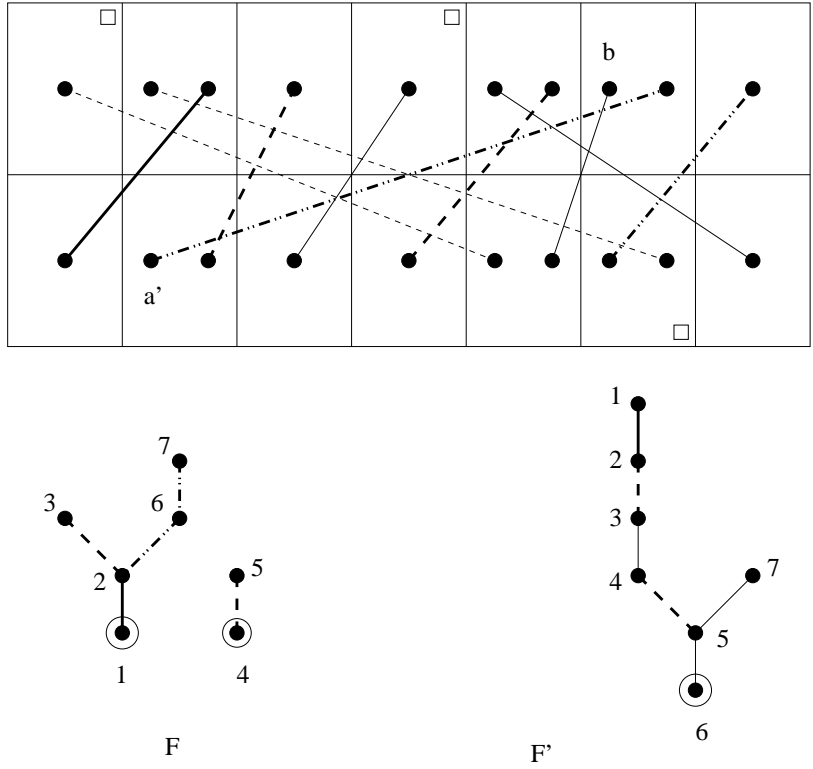}}
\end{center}
\caption{A vertical paired array with rightmost forests.}\label{verteg}
\end{figure}

For example, suppose that $\al$ is the full, vertical paired array 
given at the top of Figure~\ref{verteg}, with $s=10$, $k=7$, $i=1$, $j=0$,
shape $(1,2,1,1,2,2,1)$, and rightmost forests $F$ and $F'$ given
at the bottom of Figure~\ref{verteg}. The lines joining the pairs in $\al$ are
of various types (and the same type of line is used in the rightmost forests
when the pair corresponds to an edge in one or other of these forests): a thick
solid line indicates a shared pair in the tail, a thick dashed line a shared
pair not in the tail, a thick dashed and dotted line a pair contributing to $F$ only,
a thin solid line a pair contributing to $F'$ only, and a thin dashed line a pair that
contributes to neither of $F,F'$. When we run the inverse of the FCA to remove the
arcs $(6,2)$ and $(7,6)$ from $F$, we obtain the ordered pair $(6,2)$, and
hence obtain $a'=2'$ as indicated in Figure~\ref{verteg},
contained in column $2$.  Thus the tail is $\rho=(2,1)$, of length $l=1$.
When we run the inverse of the FCA to remove the
arcs $(2,3)$, $(3,4)$, $(4,5)$, $(5,6)$ and $(7,5)$ from $F'$, we
obtain the $5$-tuple $(3,4,5,5,6)$, and hence obtain $b=8$ as 
indicated in Figure~\ref{verteg}, contained in column $6$. Thus we conclude
that $\al$ has type $(1,(2,1), 2', 8)$.

\begin{theorem}\label{nonemvert}
For $i,j\ge 0$, $k,s\ge 1$, we have
\begin{equation*}
f^{(s)}_{k,i,j}=
s! \sum_{l=0}^{k-1} \binom{s-1-l } {k-1-l} \binom{k-1-l}{i} \binom{k-1-l}{j}.
\end{equation*}
\end{theorem}

\begin{proof}
Each paired array in $\mcF^{(s)}_{k,i,j}$ has a unique type and shape, and we
can uniquely construct those of given type and shape as follows. For the given shape, we
begin with a $2$ by $k$ array, with each cell containing an ordered set
of vertices of prescribed size.  Then we pair these vertices (all are mixed pairs)
in all possible ways for the given type in four stages. First, we
pair the rightmost vertices as prescribed by the tail. Second, we use the FCA
to pair the rightmost vertices in row $2$ (note that $b$ is safe, by construction).
Third, we use the FCA to pair the unpaired rightmost vertices in row $1$ (note
that $a'$ is safe, by construction). Fourth, we pair the remaining vertices
arbitrarily. (In addition, along the way, we have to choose the marked cells in a
consistent fashion.) In this way, for each composition $\la$ of $s$ with $k$ parts
and $l\ge 0$, we enumerate elements of $\mcF^{(s)}_{k,i,j}$ with
shape $\la$ and tail of length $l$. There are three cases:
\vspace{.1in}

\noindent
{\bf\underline{Case 1} (Vertex $a'$ is not rightmost in its cell):}
There are $s-k$ choices for $a'$, which then fixes $\rho_0$. There are
then $(k-1)_{l}$ choices for $\rho_1,\ld ,\rho_l$, and then $l$ of the pairs
are determined. This leaves $s-l$ choices for $b$ (any vertex not yet paired in
row $1$). Now mark the cell in row $1$, column $\rho_l$. Also, if $b$ is in the
tail, mark the cell in row $2$, column $\rho_0$, or if $b$ is not
in the tail, mark the cell in row $2$ of the column that contains $b$ ($b$ is
in row $1$).
        Choose, from the $k-l-1$ cells not in the tail or already marked,
        $i$ cells to mark in the top row, and $j$ cells to mark in the bottom row.
Now pair the unpaired rightmost vertices of unmarked cells in row $2$ with vertices
in row $1$ to satisfy the forest condition, using vertex $b$ as the safe position.
There are $(s-l-1)_{k-j-l-2}$ possible choices for this, from the FCA.
Suppose that there are $n$ unmarked cells in row $1$ whose rightmost vertices
are not yet paired. Then pair these with vertices in row $2$ to satisfy the
forest condition, using vertex $a'$ as the safe position. There
are $(s-k+j)_{n-1}$ possible choices for this, from the FCA. Finally,
there are $(s-k+j-n+1)!$ ways to pair the remaining vertices, arbitrarily.
But $(s-k+j)_{n-1}\cdot (s-k+j-n+1)!=(s-k+j)!$, and we conclude that
the number of elements in $\mcF^{(s)}_{k,i,j}$ with
shape $\la$ and tail of length $l$ in this case is
        \begin{equation}\label{case1claim2}
            (s-k)(k-1)_l(s-l) (s-l-1)_{k-j-l-2} (s-k+j)! \binom{k-l-1}{i} \binom{k-l-1}{j}.
        \end{equation}
(It is straightforward the check that these cardinalities are correct when we use
the non-FCA options also.)
\vspace{.1in}

\noindent
{\bf \underline{Case 2} (Vertex $a'$ is rightmost in its cell, and $b$ is in the tail):}
The number of choices for $b$ and $\rho$ is  $s\binom{k-1}{l} (l+1)! -l(k)_{l+1}$.
Then  $a'$, and $l$ of the pairs, are uniquely determined. The cell in
row $1$, column $\rho_l$ is marked. The cell
in row $2$, column $\rho_0$ is marked for two reasons -- because $b$ is in the tail,
and because the cell contains $a'$, rightmost, so
        that $a'$ will be paired with a non-dependent vertex in row $1$.
The rest of the enumeration proceeds as in Case 1, and we conclude that
the number of elements in $\mcF^{(s)}_{k,i,j}$ with
shape $\la$ and tail of length $l$ in this case is
        \begin{equation}\label{case2claim2}
            \left( s \binom{k-1}{l} (l+1)! - l(k)_{l+1} \right) (s-l-1)_{k-j-l-2} (s-k+j)! 
                \binom{k-l-1}{i} \binom{k-l-1}{j}.
        \end{equation}

\noindent
{\bf \underline{Case 3} (Vertex $a'$ is rightmost in its cell, and $b$ is not in the tail):}
The number of choices for $b$ and $\rho$ is $s(k-1)_{l+1}$, and then $a'$ together with $l$ of
the pairs are uniquely determined. In this case three different cells
must now be marked:  in
row $1$, column $\rho_l$; in row $2$, column $\rho_l$; in row $2$ of the
column that contains $b$. There are then $\binom{k-l-1}{i} \binom{k-l-2}{j-1}$ ways to
choose which other cells are marked, and the rest of the enumeration
proceeds as in Cases 1 and 2. We conclude that
the number of elements in $\mcF^{(s)}_{k,i,j}$ with
shape $\la$ and tail of length $l$ in this case is
        \begin{equation}\label{case3claim2}
            s(k-1)_{l+1} (s-l-1)_{k-j-l-2} (s-k+j)! \binom{k-l-1}{i} \binom{k-l-2}{j-1}.
        \end{equation}
        
        Adding~(\ref{case1claim2}),~(\ref{case2claim2}), and~(\ref{case3claim2}),
and simplifying,
we obtain that the total number of elements in $\mcF^{(s)}_{k,i,j}$ with
shape $\la$ and tail of length $l$ is
$$s(k-1)_{l}(s-l-1)!\binom{k-l-1}{i} \binom{k-l-1}{j},$$
and the result follows, by summing over $l\ge 0$ and multiplying
by $\binom{s-1}{k-1}$, the number of choices for $\la$.
\end{proof}

In the next result, we give an explicit enumeration for vertical
paired arrays, by applying Theorem~\ref{nonemvert}. The proof
is quite technical, involving generating functions and a hypergeometric
summation.

\begin{theorem}\label{closedform}
    For $i,j\ge 0$, $k,s\ge 1$, we have
    \begin{equation*}
        v_{k,i,j}^{(s)} =
 \f{(s+i)! (s+j)!}{(s+i+j)!}\binom{s+i+j}{k-1} \left[
            \binom{k-1}{i}\binom{k-1}{j} - \binom{k-1}{s+i}\binom{k-1}{s+j}
            \right]
    \end{equation*}
\end{theorem}

\begin{proof}
If a column in a vertical paired array has no vertices, then
at least one of the cells in rows $1$ and $2$ of that column
must be marked, from the nonempty condition. Thus, 
suppose that a vertical paired array with  $s$ mixed pairs
and $k$ columns,  has $k-m$ columns with no vertices
and $m$ with a positive number of vertices (the same number in both rows of such columns). 
Of the $k-m$ columns with no vertices, suppose that $a$ are marked
in row $1$ only, $b$ are marked in row $2$ only, and that $c-m$ are marked
in both row $1$ and $2$ (we use this parameterization for convenience
in determining the summations below).
Then we have
    \begin{equation}\label{tsum}
        v^{(s)}_{k,i,j} = \sum_{\substack{a,b,c\ge 0\\a+b+c=k}}
 \frac{k!}{a!\, b!\, c!}\, S_{a,b,c},
    \end{equation}
where
\begin{equation*}
S_{a,b,c}=
\sum_{m=0}^s\binom{c}{m}f^{(s)}_{m,i-a-c+m,j-b-c+m}.
    \end{equation*} 
But, from Theorem~\ref{nonemvert}, we have
    \begin{eqnarray*}
        S_{a,b,c}
            &=& s!\!\!\!\!
\sum_{s-1\ge m-1\ge l\ge 0} \binom{c}{m}\binom{s-1- l}{m-1- l}
                {m-1- l \choose a+c-i-1- l}{m-1- l\choose b+c-j-1- l}\\
            &=&s!\, [y^{a+c-i-1}z^{b+c-j-1}]
\sum_{s-1\ge m-1\ge l\ge 0} \binom{c}{m}\binom{s-1- l}{m-1- l}
            \big(yz\big)^{l}\big( (1+y)(1+z)\big)^{m-1-l},
    \end{eqnarray*}
where we use the notation $[A]B$ to denote the coefficient of $A$ in the expansion of $B$.
Now
    \begin{equation*}
\binom{c}{m}=\binom{-m-1}{c-m}(-1)^{c-m}
        =[x^{c}]x^m(1-x)^{-m-1},
    \end{equation*}
which gives
    \begin{equation}\label{coeffSF}
        S_{a,b,c}=s!\, [x^{c}y^{a+c}z^{b+c}]G(x,y,z),
    \end{equation}
    where, summing over $m$ by the binomial theorem, we have
    \begin{eqnarray*}
    G(x,y,z)
        &=& \f{x\, y^{i+1}z^{j+1}}{(1-x)^2}\sum_{l =0}^{s-1}
            \left( \f{xyz}{1-x}\right) ^{l}
            \left( 1+\f{x(1+y)(1+z)}{1-x}\right)^{s-l -1}\\
        &=&\f{x\, y^{i+1}z^{j+1}}{(1-x)^2}
            \f{\left(1+\f{x(1+y)(1+z)}{1-x} \right)^{s}-\left(\f{xyz}{1-x}\right) ^s}{1+\f{x(1+y)(
            1+z)}{1-x}-\f{xyz}{1-x}}\\
        &=&\f{x\, y^{i+1}z^{j+1}}{(1-x)^{s+1}}
            \f{\big( 1+x(y+z+yz)\big) ^s-\big( xyz\big) ^s}{1+x(y+z)}.
    \end{eqnarray*}
    But, changing variables in~(\ref{coeffSF}), we have
    \begin{equation*}
        S_{a,b,c}=s!\, [x^{0}y^{a+c}z^{b+c}]x^k G(x,\f{y}{x},\f{z}{x}),
    \end{equation*}
    so from~(\ref{tsum}) we obtain
    \begin{eqnarray*}
        t^{(s)}_{k,i,j}
            &=&s!\, [x^0y^{k}z^{k}]x^k (1+y+z)^k G(x,\f{y}{x},\f{z}{x})\\
            &=&s!\, [x^0y^{k}z^{k}]x^{k-i-j-s-1}y^{i+1}z^{j+1}\f{(1+y+z)^{k-1}}{(1-x)^{s+1}}
                \Big(\big( x(1+y+z)+yz\big) ^s-\big( yz\big) ^s\Big)\\
            &=&R_1-R_2,
    \end{eqnarray*}
    where
    \begin{eqnarray*}
        R_2&=&s!\, [x^{i+j+s-k+1}y^{k-i-1}z^{k-j-1}]\f{(1+y+z)^{k-1}}{(1-x)^{s+1}}
            (yz)^s\\                      
        &=&s!{2s+i+j-k+1\choose s}\f{(k-1)!}{(k-s-i-1)!(k-s-j-1)!(2s+i+j-k+1)!}\\
        &=& \f{(s+i)! (s+j)!}{(s+i+j)!}{s+i+j \choose k-1}
            {k-1 \choose s+i}{k-1 \choose s+j},\\
            R_1&=&s!\, [x^{i+j+s-k+1}y^{k-i-1}z^{k-j-1}]\f{(1+y+z)^{k-1}}{(1-x)^{s+1}}
            \big( x(1+y+z)+yz\big) ^s\\
        &=&s!\, [x^{i+j+s-k+1}y^{k-i-1}z^{k-j-1}]\sum_{m\ge 0}{s\choose m}
            x^{s-m}(1+y+z)^{k+s-m-1}(yz)^{m}(1-x)^{-s-1}\\
        &=&s!\sum_{m\ge 0}{s\choose m}{s+i+j+m-k+1\choose s}
            \f{(s-m+k-1)!}{(k-m-i-1)!(k-m-j-1)!(s+i+j+m-k+1)!}.
\end{eqnarray*}
Now, for this latter sum over $m\ge 0$, we observe that the ratio of
the $m+1$st term to the $m$th term is a rational function of $m$, which
implies that it is a hypergeometric sum. In
particular, using the standard notation for hypergeometric series, we have
\begin{eqnarray*}
R_1&=&\f{(s+k-1)!}{(k-i-1)! (k-j-1)! (i+j-k+1)!}\, {}_{3}F_{2}
 \left(  \begin{array}{c} i+1-k, \ j+1-k, \ -s \\
i+j-k+2, \ 1-s-k
\end{array}  ; 1 \right) \\
        &=&\f{(s+k-1)!}{(k-i-1)! (k-j-1)! (i+j-k+1)!}
            \f{{s+i \choose s}{s+j \choose s}}{{s+k-1 \choose s}{s+i+j-k+1 \choose s}}\\
        &=&  \f{(s+i)! (s+j)!}{(s+i+j)!}{s+i+j \choose k-1}
            {k-1 \choose i}{k-1 \choose j},
    \end{eqnarray*}
    where the second last equality follows from the 
    Pfaff-Saalsch\"utz Theorem for ${}_3 F_2$ hypergeometric summations 
(see, \textit{e.g.}, Theorem 2.2.6 on
page 69 of~[\ref{aar}]).
    The result follows immediately.
\end{proof}

We finish with the proof of our main result, now that we have completed
all the intermediate results in our reduction.
\vspace{.1in}

\noindent
\textit{Proof of Theorem~\ref{mainthm}}. From Proposition~\ref{fallfac},  
~(\ref{equivbe}) and Theorems~\ref{redundmin},~\ref{nonmixedprop}, we obtain 
$$A^{(s)}_{p,q}(x)= \sum_{k\ge 1}{x\choose k}\sum_{i,j\ge 0} 
{p\choose 2i}(2i-1)!!{q\choose 2j}(2j-1)!!
 (p-2i)_{\tfh(p-s-2i)}(q-2j)_{\tfh(q-s-2j)}v^{(s)}_{k,\tfh(p-s-2i),\tfh(q-s-2j)}.$$
But, simplifying, we obtain
$$ {p\choose 2i}(2i-1)!!{q\choose 2j}(2j-1)!!
 (p-2i)_{\tfh(p-s-2i)}(q-2j)_{\tfh(q-s-2j)}=
\f{p!q!}{2^{i+j}i!j!(\tfh(p+s)-i)!(\tfh(q+s)-j)!}
,$$
and the result follows from
Theorem~\ref{closedform}.
\hfill{\Large$\Box$}

\section*{Acknowledgements}

This work was supported by a Discovery Grant from NSERC (IG)
and by a USRA Award from NSERC (WS). We are grateful to Andu
Nica for helpful discussions and suggestions. 

\section*{References}

\begin{enumerate}

\item\label{aar}
G. E. Andrews, R. Askey and R. Roy, ``Special functions'', Encyclopedia
of Mathematics and its Applications, Vol. 71, Cambridge University Press, 1999.

\item\label{ht}
U. Haagerup and S. Thorbjornsen, {\em Random matrices with 
complex Gaussian entries},
Expositiones Math.,
{\bf 21} (2003), 293--337.

\item\label{gn}
I. P. Goulden and A. Nica, {\em A direct bijection for the Harer-Zagier
formula},
J. Comb. Theory (A),
{\bf 111} (2005), 224--238.

\item\label{hz}
J. Harer and D. Zagier, {\em The Euler characteristic
of the moduli space of curves},
Inventiones Mathematicae,
{\bf 85} (1986), 457--486.

\item\label{iz}
C. Itzykson and J.-B. Zuber, {\em Matrix integration and combinatorics of modular groups},
Comm. Math. Phys.,
{\bf 134} (1990), 197--207.

\item\label{j}
D. M. Jackson, {\em On an integral representation for the genus series
for $2$-cell embeddings},
Trans. Amer. Math. Soc.,
{\bf 344} (1994), 755--772.

\item\label{ke}
S. Kerov, {\em Rook placements on Ferrers boards, and matrix integrals} 
(English translation), J. Math. Sci. (New York),
{\bf 96} (1999), 3531--3536.

\item\label{k}
M. Kontsevich, {\em Intersection theory on the moduli space of curves
and the matrix Airy function},
Comm. Math. Phys.,
{\bf 147} (1992), 1--23.

\item\label{lz}
S. K. Lando and A. K. Zvonkin,
``Graphs on surfaces and their applications'', Springer-Verlag,
Berlin, Heidelberg, 2004.

\item\label{l}
B. Lass, {\em D\'emonstration combinatoire de la formule
de Harer-Zagier},
C. R. Acad. Sci. Paris,
{\bf 333}, S\'erie I (2001), 155--160.

\item\label{mn}
J. A. Mingo and A. Nica, {\em Annular noncrossing permutations and partitions, and
second-order asymptotics for random matrices},
International Math. Research Notices,
no. 28 (2004), 1413 -- 1460.

\item\label{p}
R. C. Penner, {\em Perturbative series and the moduli space of Riemann surfaces},
J. Diff. Geometry,
{\bf 27} (1988), 35--53.

\item\label{stan}
R. P. Stanley, ``Enumerative Combinatorics'', Vol. 2, Cambridge University Press, 1999.


\item\label{t}
W. T. Tutte, ``Graph Theory'', Addison-Wesley, London, 1984.

\item\label{z}
D. Zagier, {\em On the distribution of the number of cycles of elements
in symmetric groups},
Nieuw Archief voor Wiskunde,
{\bf 13} (1995), 489--495.

\item\label{zv}
A. Zvonkin, {\em Matrix integrals and map enumeration: An accessible introduction},
Combinatorics and Physics (Marseilles, 1995),  Math. Comput. Modelling,
{\bf 26} (1997), 281--304.

\end{enumerate}

\end{document}